\documentclass{ourlematema}
\usepackage{tabularx}
\usepackage{amsmath, amssymb, amsthm, mathtools} \usepackage{tabularx, array}
\usepackage[dvipsnames]{xcolor}
\usepackage{tikz}
\usetikzlibrary{graphs,trees,arrows,positioning,shapes.geometric}
\usepackage{tikzscale}
\usepackage{array}
\newcolumntype{C}{>{$}c<{$}}

\usepackage{float}

\newcommand{\highlight}[1]{%
  \colorbox{yellow!50}{$\displaystyle#1$}}
\newcommand{\mld}{\operatorname{mld}}
\newcommand{\rmld}{\operatorname{rmld}}
\newcommand{\Aut}{\operatorname{Aut}}

\title{Coloured graphical models \\ and their symmetries}
\titlemark{Coloured graphical models and their symmetries}
\author{Isobel Davies}
\address{Max Planck Institute for Mathematics in the Sciences\\ Inselstraße 22, 04103 Leipzig\\
\email{davies@mis.mpg.de}}

\author{Orlando Marigliano}
\address{KTH Royal Institute of Technology\\
Lindstedtsvägen 25, 10044 Stockholm, Sweden\\
\email{orlandom@kth.se}}



\begin{document}

\maketitle

\begin{abstract} 
Coloured graphical models are Gaussian statistical models determined by an undirected coloured graph.
These models can be described by linear spaces of symmetric matrices. 
We outline a relationship between the symmetries of the graph and the linear forms that vanish on the reciprocal variety of the model. In particular, we give four families for which such linear forms are completely described by symmetries.
\end{abstract}

\section{Introduction}
Coloured graphical models are a class of Gaussian statistical models first introduced in~\cite{lauritzen}
to model situations where some of the covariates have approximately equal empirical concentrations.
These equalities are represented by a coloured graph.
The set of concentration matrices associated to a coloured graphical model is a linear subspace
of the set of real symmetric $n\times n$ matrices.

Following the approach of algebraic statistics we view  maximum likelihood (ML) estimation as a problem in algebraic geometry, see for instance~\cite[Ch.\ 7]{sullivant}. In our case,
we consider the complex linear space of symmetric matrices $\mathcal L\subseteq \mathbb S^n$ defined by a given coloured graphical model and its 
\emph{reciprocal variety} $\mathcal L^{-1}$, which is defined by the matrix inverses of the points in $\mathcal L$.
In particular, algebraic statisticians are interested in a collection of algebraic properties of $\mathcal L$, such as its ML degree $\mld(\mathcal L)$, its reciprocal ML degree $\rmld(\mathcal L)$, the degree $\deg(\mathcal L^{-1})$ and the ideal $I(\mathcal L^{-1})$ defining the reciprocal variety. For the definitions of these properties, see~\cite{bernd-claudia}.

In~\cite{bernd-caroline}, Sturmfels and Uhler
compute some of these algebraic properties for coloured graphical models whose underlying graph is the four-cycle.
Inspired by their work, we are interested in understanding these properties for coloured graphical models.
In this article, we explore the notion of symmetry of a coloured graph and study its connection with the linear part of the ideal $I(\mathcal L^{-1})$. In particular, we focus on \emph{uniform coloured graphs}, whose symmetries coincide with the usual graph symmetries of the underlying uncoloured graph.
For a number of families of uniform coloured graphs we completely describe the linear part of the ideal $I(\mathcal L^{-1})$ via symmetries. For this, we use the results in~\cite{bernd-claudia} on pencils of quadrics.

\paragraph{Notation.}
Let $G = (V,E)$ be a simple undirected graph, let $n = |V|$ and $E^c$ denote the set of edges of the graph complement $G^c$. The (uncoloured) \emph{adjacency matrix} of $G$ is the $n\times n$ matrix whose $(i,j)$-th entry is $1$ if $(i,j)\in E$ and $0$ otherwise.
A \emph{colouring} of $G$ is a partition of $V$ together with a partition of $E$.
A \emph{colour} is an equivalence class with respect to one of these partitions. 

Let $G$ be a coloured graph with $n$ vertices and $d$ colours $\gamma_1,\dotsc,\gamma_d$. For $k=1,\dotsc,d$ we define the matrix $A_k$ associated to the colour $\gamma_k$ as follows. If $\gamma_k$ is a vertex colour then $(A_k)_{ii} = 1$ if vertex $i$ has the colour $\gamma_k$, while the other entries are all zero. If $\gamma_k$ is an edge colour then $(A_k)_{ij} = 1$ if edge $(i,j)$ exists and has the colour $\gamma_k$, and all the other entries are zero.
The (coloured) \emph{adjacency matrix} $A$ of $G$ is the matrix
\[
\sum_{k=1}^d \lambda_k A_k
\in \mathbb Z[\lambda_1,\dotsc,\lambda_d]^{n\times n}.
\]

Let $G$ be a coloured graph with $n$ vertices and $d$ colours. The associated linear space of symmetric matrices $\mathcal L$ is the subspace
\[
\left\{\sum_{k=1}^d \lambda_k A_k : \lambda_1,\dotsc,\lambda_d\in \mathbb C\right\} 
\]
of the space of complex symmetric $n\times n$ matrices $\mathbb S^n$. The \emph{reciprocal variety} $\mathcal L^{-1}$ is the Zariski-closure of the set
\[
\{A^{-1} : A\in \mathcal L \text{ invertible}\}\subseteq \mathbb S^n.
\]
Its ideal is denoted by $I(\mathcal L^{-1})$.

\section{Linear forms from symmetries}

We start by defining graph symmetries, following the approach in~\cite[Sec.\ 5]{lauritzen}. 

\begin{dfn}
Let $A$ be the uncoloured, resp.\ coloured adjacency matrix of an uncoloured, resp.\ coloured graph $G$. A \emph{symmetry} of $G$ is an $n\times n$ permutation matrix $B$ such that
$BAB^{-1} = A$.
\end{dfn}

Such a symmetry gives a set of linear forms that vanish on the reciprocal variety $\mathcal L^{-1}$: 

\begin{prop}\label{symmetries}
Let $G$ be a coloured graph, $B$ a symmetry of $G$ and $X$ a generic element of $\mathbb S^n$. The binomial linear forms defined by all distinct entries of the matrix $BXB^{-1} - X$ belong to $I(\mathcal L^{-1})$.
\end{prop}

\begin{proof}
If $B$ is a symmetry of $G$, then $BA^{-1}B^{-1} = A^{-1}$ in $\mathbb Q(\lambda_1,\dotsc,\lambda_d)^{n\times n}$ by taking inverses. Thus, a generic element $X $ of $\mathcal L^{-1}$ must satisfy the equation $BXB^{-1} - X = 0,$ whose entries are binomial linear forms.
\end{proof}

We call the linear forms that come from Proposition~\ref{symmetries} \emph{induced by symmetries}. When the linear part of $I(\mathcal L^{-1})$ is generated by these linear forms, we also say it is \emph{induced by symmetries}.

\begin{exa}\label{five-cycle}
Let $G$ be the following coloured $5$-cycle:
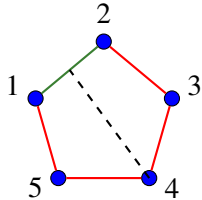
\begin{center}
\begin{tikzpicture}
\begin{scope}[scale=1.5]
\draw (0.2,0) node[draw,circle,inner sep=2pt,fill=blue, label={[xshift=-0.3cm,yshift=-0.5cm]5}](5){};
\draw (0,0.7) node[draw,circle,inner sep=2pt,fill=blue, label={[xshift=-0.3cm,yshift=-0.2cm]1}](1){};
\draw (1,0) node[draw,circle,inner sep=2pt,fill=blue, label ={[xshift=0.3cm, yshift=-0.5cm]4}](4){};
\draw (1.2,0.7) node[draw,circle,inner sep=2pt,fill=blue, label ={[xshift=0.3cm,yshift=-0.2cm]3}](3){};
\draw (0.6,1.2) node[draw,circle,inner sep=2pt,fill=blue,label={2}](2){};
\draw[-] (1) edge[OliveGreen,thick] (2);
\draw[-] (2) edge[red,thick] (3);
\draw[-] (3) edge[red,thick] (4);
\draw[-] (4) edge[red,thick] (5);
\draw[-] (5) edge[red,thick] (1);
\draw[dashed] (1,0) edge[black,thick] (0.3,.95);
\end{scope}
\end{tikzpicture}
\end{center}
The only nontrivial symmetry of $G$ is the reflection $B_\tau$ along the depicted axis.
The matrix $B_\tau X B_\tau^{-1} - X$ equals
\[
\begin{pmatrix}
-x_{11}+x_{22}& 0& -x_{13}+x_{25}& -x_{14}+x_{24}& -x_{15}+x_{23}\\
0& x_{11}-x_{22}& x_{15}-x_{23}& x_{14}-x_{24}& x_{13}-x_{25}\\
-x_{13}+x_{25}& x_{15}-x_{23}& -x_{33}+x_{55}& -x_{34}+x_{45}& 0\\
-x_{14}+x_{24}& x_{14}-x_{24}& -x_{34}+x_{45}& 0& x_{34}-x_{45}\\
-x_{15}+x_{23}& x_{13}-x_{25}& 0& x_{34}-x_{45}& x_{33}-x_{55}
\end{pmatrix}.
\]
Hence, the linear forms induced by symmetries in $I(\mathcal L^{-1})$ are:
\[
x_{11} - x_{22}, \quad x_{33} - x_{55}, \quad x_{13} - x_{25}, \quad
x_{14} - x_{24}, \quad x_{15} - x_{23}, \quad x_{34} - x_{25}.
\]
These do not generate the whole linear part of $I(\mathcal L^{-1})$. For instance, the linear form $x_{14}+x_{44}-x_{35}-x_{55}$ is contained in the ideal but is not a linear combination of the above.
\end{exa}
\begin{exa}\label{example3}
Let $G$ be the following disconnected graph:
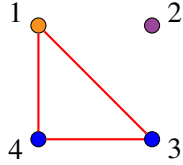
\begin{center}
\begin{tikzpicture}
\begin{scope}[scale=1.5]
    \draw (0,0) node[draw,circle,inner sep=2pt,fill=blue, label={[xshift=-0.3cm,yshift=-0.5cm]4}](4){};
  \draw (0,1) node[draw,circle,inner sep=2pt,fill=BurntOrange, label={[xshift=-0.3cm,yshift=-0.2cm]1}](1){};
  \draw (1,0) node[draw,circle,inner sep=2pt,fill=blue, label ={[xshift=0.3cm, yshift=-0.5cm]3}](3){};
  \draw (1,1) node[draw,circle,inner sep=2pt,fill=Purple, label ={[xshift=0.3cm,yshift=-0.2cm]2}](2){};
  \draw[-] (1) edge[red,thick] (3);
  \draw[-] (3) edge[red,thick] (4);
  \draw[-] (1) edge[red,thick] (4);
  \end{scope}
 \end{tikzpicture}
\end{center}
The only nontrivial symmetry of $G$ is the permutation $B_\rho$ of vertices $3$ and $4$. The matrix $B_\rho XB_\rho^{-1}-X$ gives all the binomial linear forms in the following list of generators of the linear part of $I(\mathcal{L}^{-1})$:
\[
x_{33} - x_{44}, \quad x_{13} - x_{14}, \quad x_{12}, \quad
x_{23}, \quad x_{24}.
\]
The following proposition explains the three other generators.
\end{exa}

\begin{prop}\label{isolated}
Let $G$ be a coloured graph and
let $i,j$ be a pair of vertices belonging to different connected components of $G$. Then we have
$x_{ij}\in I(\mathcal L^{-1})$.
\end{prop}

\begin{proof}
The adjacency matrix $A$ of $G$ can be written as a block diagonal matrix whose blocks are the adjacency matrices of the connected components of $G$. Since its inverse $A^{-1}$ retains this block structure, a generic element $X$ of $\mathcal L^{-1}$ must have a zero entry at each pair of indices $(i,j)$ as in the statement.
\end{proof}


Let $I'$ be the ideal generated by all linear forms induced by symmetries of $G$ and all linear forms found by applying Proposition~\ref{isolated}. We define a further linear subspace $\mathcal L'$ as the vanishing set $V(I')$. Like $\mathcal L$, this linear subspace is associated to a coloured graph $G'$, whose colouring is induced by the linear forms in $I'$.
In
~\cite[Sec.\ 5]{lauritzen}, the linear subspace $\mathcal L'$ is
called an \emph{RCOP model}.

\begin{exa}\label{g-prime}
The linear forms found in Examples~\ref{five-cycle} and~\ref{example3} give rise to the following graphs $G_1'$ and $G_2'$. 
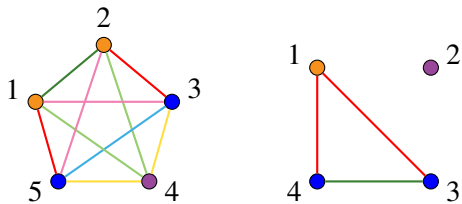
\begin{center}
\begin{tikzpicture}
\begin{scope}[scale=1.5]
\draw (0.2,0) node[draw,circle,inner sep=2pt,fill=blue, label={[xshift=-0.3cm,yshift=-0.5cm]5}](5){};
\draw (0,0.7) node[draw,circle,inner sep=2pt,fill=BurntOrange, label={[xshift=-0.3cm,yshift=-0.2cm]1}](1){};
\draw (1,0) node[draw,circle,inner sep=2pt,fill=Purple, label ={[xshift=0.3cm, yshift=-0.5cm]4}](4){};
\draw (1.2,0.7) node[draw,circle,inner sep=2pt,fill=blue, label ={[xshift=0.3cm,yshift=-0.2cm]3}](3){};
\draw (0.6,1.2) node[draw,circle,inner sep=2pt,fill=BurntOrange,label={2}](2){};
\draw[-] (1) edge[OliveGreen,thick] (2);
\draw[-] (2) edge[red,thick] (3);
\draw[-] (3) edge[Goldenrod,thick] (4);
\draw[-] (4) edge[Goldenrod,thick] (5);
\draw[-] (5) edge[red,thick] (1);
\draw[-] (2) edge[YellowGreen,thick] (4);
\draw[-] (3) edge[CornflowerBlue,thick] (5);
\draw[-] (4) edge[YellowGreen,thick] (1);
\draw[-] (3) edge[CarnationPink,thick] (1);
\draw[-] (2) edge[CarnationPink,thick] (5);
\end{scope}
\end{tikzpicture}
\qquad
\begin{tikzpicture}
\begin{scope}[scale=1.5]
    \draw (0,0) node[draw,circle,inner sep=2pt,fill=blue, label={[xshift=-0.3cm,yshift=-0.5cm]4}](4){};
  \draw (0,1) node[draw,circle,inner sep=2pt,fill=BurntOrange, label={[xshift=-0.3cm,yshift=-0.2cm]1}](1){};
  \draw (1,0) node[draw,circle,inner sep=2pt,fill=blue, label ={[xshift=0.3cm, yshift=-0.5cm]3}](3){};
  \draw (1,1) node[draw,circle,inner sep=2pt,fill=Purple, label ={[xshift=0.3cm,yshift=-0.2cm]2}](2){};
  \draw[-] (1) edge[red,thick] (3);
  \draw[-] (3) edge[OliveGreen,thick] (4);
  \draw[-] (1) edge[red,thick] (4);
  \end{scope}
 \end{tikzpicture}
 \end{center}
\end{exa}

By Proposition~\ref{symmetries}, every symmetry of a coloured graph $G$ is also a symmetry of $G'$. We also see that both $\mathcal L$ and $\mathcal L^{-1}$ are contained in $\mathcal L'$.
The next proposition demonstrates how $\mathcal L'$ can be used as a lower-dimensional ambient space to compute the ML degree of $\mathcal L$.

\begin{prop}
Let $S'$ be a general matrix of $\mathcal L'$ and let $(\mathcal L^{\perp})' = \mathcal L^\perp\cap \mathcal L'$, where $\mathcal L^\perp\subseteq \mathbb S^n$ is the orthogonal complement to $\mathcal L$ with respect to the trace inner product. Let $(\mathcal L^\perp)' + S'$ denote the affine space $\{X' + S'\mid X'\in (\mathcal L^\perp)'\}$.
The ML degree of $\mathcal L$ is the cardinality of the set
\[
\mathcal L^{-1}\cap ((\mathcal L^\perp)' + S').
\]
\end{prop}
\begin{proof}
By definition, the ML degree of $\mathcal L$ is the cardinality of
$\mathcal L^{-1} \cap
(\mathcal L^\perp + S)$,
where $S$ is a general matrix of $\mathbb S^n$ and $\mathcal L^\perp\subseteq \mathbb S^n$ is the orthogonal complement of $\mathcal L$ in $\mathbb S^n$.
Since $\mathcal L\subseteq \mathcal L'$, we have $\operatorname{span}(\mathcal L', \mathcal L^\perp) = \mathbb S^n$.
We write all matrices $M\in \mathbb S^n$ as $M = M_1 + M_2 + M_3$, where the term $M_1$ is the part of $M$ contained in $(\mathcal L^\perp)'$, the term $M_2$ is the remaining part contained in $\mathcal L^\perp$, and $M_3$ is the remaining part contained in $\mathcal L'$.
Let $S' = S_1 + S_3.$ Since $S$ is general in $\mathbb S^n$, the matrix $S'$ is general in $\mathcal L'$, so it remains to show that
\[
(\mathcal L^\perp + S) \cap \mathcal L' = (\mathcal L^\perp)' + S'.
\]
The reverse inclusion is clear. For the forward inclusion,
Let $\Sigma \in (\mathcal L^\perp + S) \cap \mathcal L'$ and write $\Sigma = X + S$ with $X\in \mathcal L^\perp$. Write $\Sigma = X_1 + X_2 + S' + S_2$ using the above decomposition for $X$ and $S$. Since $\Sigma\in \mathcal L',$ we have $X_2 + S_2\in \mathcal L'$. By construction, $X_2 + S_2 = 0$, so $\Sigma = X_1 + S'\in (\mathcal L^\perp)' + S'$, as required.
\end{proof}


\section{Uniform coloured graphs}\label{uniform-coloured-graphs}

In this section, we apply Proposition~\ref{symmetries} to coloured graphs whose colouring has only one edge colour and one vertex colour.
Such a graph is uniquely determined by its underlying uncoloured graph, and we call it a
\emph{uniform coloured graph}. Its group of symmetries coincides with the group of symmetries of its uncoloured graph. We are interested in the question: when is the linear part of $I(\mathcal L^{-1})$ induced by symmetries?

The adjacency matrix of a uniform coloured graph has the form
\[
\lambda_1A_1 + \lambda_2A_2
\]
where $A_1$ is the $n\times n$ identity matrix and $A_2$ is the adjacency matrix of the underlying uncoloured graph.
Since the corresponding linear space $\mathcal L$ is two-dimensional, 
we may apply the results in \cite{bernd-claudia} on pencils of quadrics. 

\begin{dfn}[\cite{bernd-claudia}]
Let $\mathcal{L}$ be a linear space spanned by two symmetric matrices $A_1$ and $A_2$, such that $A_1$ is an invertible matrix. The \emph{Segre symbol} of $\mathcal{L}$ is the unordered list
\[
\sigma=\left[\sigma_1\ldots\sigma_r\right]
\]
specified by the Jordan canonical form of $A_1^{-1}A_2$
in the following way.
Let $\alpha_1,\dotsc,\alpha_r$ be the distinct eigenvalues of $A_1^{-1}A_2$.
Each $\sigma_i$ is itself a tuple of natural numbers associated to the eigenvalue $\alpha_i$.
The entries of $\sigma_i$ are precisely the sizes of the Jordan blocks associated to $\alpha_i$.
\end{dfn}

If $\mathcal L$ is the linear space associated to a uniform coloured graph, then its Segre symbol is a description of the spectrum of the underlying uncoloured graph.
More precisely, its entries $\sigma_i$ correspond to the distinct eigenvalues of $A_2$ and are of the form $(1,\dotsc,1)$, where the number of ones equals the multiplicity of the eigenvalue.

\begin{exa}
The spectrum of the Petersen graph is $(-2_{4}, 1_{5}, 3)$.
Thus, the linear space associated to its uniform coloured version
has Segre symbol $[1_4,1_5,1]$.
\end{exa}

Using Segre symbols, it follows by~\cite[Thms.\ 3.2, 4.2]{bernd-claudia} that
the number $r$ of distinct eigenvalues of $A_2$
determines certain algebraic properties of $\mathcal L$. We summarize these in Table~\ref{tablealgprop}, along with the specific values for the Petersen graph.
\begin{table}[htbp]
\begin{center}
\renewcommand*{\arraystretch}{1.5}
\begin{tabular}{|c|c|c|c|}
\hline
& Uniform Coloured Graph & Petersen Graph\\
\hline
$\deg(\mathcal L^{-1})$  &$r-1$& 2\\
$\operatorname{mld}(\mathcal L)$  &$r-1$& 2\\
$\operatorname{rmld}(\mathcal L)$  &$2r-3$ & 3\\
no. linear forms
 &$\binom{n+1}{2}-r$&52\\
no. quadratic forms
 & $\binom{r-1}{2}$&1 \\[5pt]
\hline
\end{tabular}
\renewcommand*{\arraystretch}{1}
\end{center}
\caption{
Algebraic properties of the linear space of symmetric matrices $\mathcal L$ associated to a uniform coloured graph and the uniform coloured Petersen graph. Here, $r$ denotes the number of distinct eigenvalues of the uncoloured adjacency matrix of the graph.
The first row is the degree of $\mathcal L^{-1}$.
The second and third rows are the ML and reciprocal ML degrees of $\mathcal L$ as defined by \cite{bernd-claudia}.
Since $\mathcal L$ is a pencil, the ideal $I(\mathcal L^{-1})$ is generated by linear and quadratic forms.
The last two rows give the number of these forms in a set of minimal generators of $I(\mathcal L^{-1})$.
}\label{tablealgprop}
\end{table}

Comparing the number of linearly independent linear forms induced by symmetries with the total number $\binom{n+1}{2}-r$, we can determine whether the linear part of $I(\mathcal L^{-1})$ is induced by symmetries for any given graph $G$. To this end, the following reformulation of Proposition~\ref{symmetries} is useful:

\begin{cor}\label{action}
Let $G$ be a coloured graph, $i,j$ two vertices, and $B\in \Aut(G)$ be a symmetry, considered as a permutation of the vertex set
$V$. Then $x_{i,j} - x_{B(i), B(j)}$ lies in $I(\mathcal L^{-1})$.
In particular,
let $V * V$ be the set of unordered pairs of vertices of $G$ and $s$ the number of orbits of the entrywise action of $\Aut(G)$ on $V * V$. Then
the number of linear forms of $I(\mathcal L^{-1})$ induced by symmetries is $\binom{n+1}{2} - s$.
\end{cor}
\begin{proof}
The statements follow from Proposition~\ref{symmetries}. The first statement follows by computing the $(B(i),B(j))$-th entry of $BXB^{-1}$.
The second follows from the fact that
every orbit $Y$ of the group action gives $|Y|-1$ linear forms, while
the number of
elements in $V * V$
is $\binom{n+1}{2}$. 
\end{proof}

It follows that the linear part of $I(\mathcal L^{-1})$ is induced by symmetries if and only if $s = r$.
Our main result gives four families of uniform coloured graphs where this is the case, illustrated in Figure~\ref{families}. We use the following conventions: the complete bipartite graph $K_{m,m}$ is split by the parity of its vertices, and the hyperoctahedral graph $H_m$ is obtained from the complete graph $K_{2m}$ by removing the edges $\{(2k-1,2k)\mid k=1,\dotsc,m\}$.

\begin{figure}[ht]
\centering
\begin{tikzpicture}
\begin{scope}[scale=1]
\footnotesize
    \draw (0,0) node[draw,circle,inner sep=2pt,fill=blue, label={[xshift=-0.3cm,yshift=-0.5cm]5}](5){};
  \draw (0,1.5) node[draw,circle,inner sep=2pt,fill=blue, label={[xshift=-0.3cm,yshift=-0.2cm]1}](1){};
  \draw (0.9,0) node[draw,circle,inner sep=2pt,fill=blue, label ={[xshift=0.3cm, yshift=-0.5cm]4}](4){};
  \draw (0.9,1.5) node[draw,circle,inner sep=2pt,fill=blue, label ={[xshift=0.3cm,yshift=-0.2cm]2}](2){};
  \draw (-0.4,0.75) node[draw,circle,inner sep=2pt,fill=blue, label ={[xshift=-0.3cm, yshift=-0.3cm]6}](6){};
  \draw (1.3,0.75) node[draw,circle,inner sep=2pt,fill=blue, label ={[xshift=0.3cm,yshift=-0.3cm]3}](3){};
  \draw[-] (1) edge[red,thick] (2);
  \draw[-] (2) edge[red,thick] (3);
  \draw[-] (3) edge[red,thick] (4);
  \draw[-] (4) edge[red,thick] (5);
  \draw[-] (5) edge[red,thick] (6);
  \draw[-] (6) edge[red,thick] (1);
  \end{scope}
\end{tikzpicture}\hspace{1em}
\begin{tikzpicture}
\begin{scope}[scale=1]
\footnotesize
    \draw (0,0) node[draw,circle,inner sep=2pt,fill=blue, label={[xshift=-0.3cm,yshift=-0.5cm]5}](5){};
  \draw (0,1.5) node[draw,circle,inner sep=2pt,fill=blue, label={[xshift=-0.3cm,yshift=-0.2cm]1}](1){};
  \draw (0.9,0) node[draw,circle,inner sep=2pt,fill=blue, label ={[xshift=0.3cm, yshift=-0.5cm]4}](4){};
  \draw (0.9,1.5) node[draw,circle,inner sep=2pt,fill=blue, label ={[xshift=0.3cm,yshift=-0.2cm]2}](2){};
  \draw (-0.4,0.75) node[draw,circle,inner sep=2pt,fill=blue, label ={[xshift=-0.3cm, yshift=-0.3cm]6}](6){};
  \draw (1.3,0.75) node[draw,circle,inner sep=2pt,fill=blue, label ={[xshift=0.3cm,yshift=-0.3cm]3}](3){};
  \draw[-] (1) edge[red,thick] (2);
  \draw[-] (2) edge[red,thick] (3);
  \draw[-] (3) edge[red,thick] (4);
  \draw[-] (4) edge[red,thick] (5);
  \draw[-] (5) edge[red,thick] (6);
  \draw[-] (6) edge[red,thick] (1);
  \draw[-] (1) edge[red,thick] (3);
  \draw[-] (2) edge[red,thick] (4);
  \draw[-] (3) edge[red,thick] (5);
  \draw[-] (4) edge[red,thick] (6);
  \draw[-] (5) edge[red,thick] (1);
  \draw[-] (6) edge[red,thick] (2);
  \draw[-] (1) edge[red,thick] (4);
  \draw[-] (2) edge[red,thick] (5);
  \draw[-] (3) edge[red,thick] (6);
  \end{scope}
\end{tikzpicture}\hspace{1em}
\begin{tikzpicture}
\begin{scope}[scale=1]
\footnotesize
    \draw (0,0) node[draw,circle,inner sep=2pt,fill=blue, label={[xshift=-0.3cm,yshift=-0.5cm]5}](5){};
  \draw (0,1.5) node[draw,circle,inner sep=2pt,fill=blue, label={[xshift=-0.3cm,yshift=-0.2cm]1}](1){};
  \draw (0.9,0) node[draw,circle,inner sep=2pt,fill=blue, label ={[xshift=0.3cm, yshift=-0.5cm]4}](4){};
  \draw (0.9,1.5) node[draw,circle,inner sep=2pt,fill=blue, label ={[xshift=0.3cm,yshift=-0.2cm]2}](2){};
  \draw (-0.4,0.75) node[draw,circle,inner sep=2pt,fill=blue, label ={[xshift=-0.3cm, yshift=-0.3cm]6}](6){};
  \draw (1.3,0.75) node[draw,circle,inner sep=2pt,fill=blue, label ={[xshift=0.3cm,yshift=-0.3cm]3}](3){};
  \draw[-] (3) edge[red,thick] (4);
  \draw[-] (6) edge[red,thick] (1);
  \draw[-] (5) edge[red,thick] (4);
  \draw[-] (3) edge[red,thick] (2);
  \draw[-] (2) edge[red,thick] (1);
  \draw[-] (6) edge[red,thick] (5);
  \draw[-] (1) edge[red,thick] (4);
  \draw[-] (2) edge[red,thick] (5);
  \draw[-] (3) edge[red,thick] (6);
  \end{scope}
\end{tikzpicture}\hspace{1em}
\begin{tikzpicture}
\begin{scope}[scale=1]
\footnotesize
    \draw (0,0) node[draw,circle,inner sep=2pt,fill=blue, label={[xshift=-0.3cm,yshift=-0.5cm]5}](5){};
  \draw (0,1.5) node[draw,circle,inner sep=2pt,fill=blue, label={[xshift=-0.3cm,yshift=-0.2cm]1}](1){};
  \draw (0.9,0) node[draw,circle,inner sep=2pt,fill=blue, label ={[xshift=0.3cm, yshift=-0.5cm]4}](4){};
  \draw (0.9,1.5) node[draw,circle,inner sep=2pt,fill=blue, label ={[xshift=0.3cm,yshift=-0.2cm]2}](2){};
  \draw (-0.4,0.75) node[draw,circle,inner sep=2pt,fill=blue, label ={[xshift=-0.3cm, yshift=-0.3cm]6}](6){};
  \draw (1.3,0.75) node[draw,circle,inner sep=2pt,fill=blue, label ={[xshift=0.3cm,yshift=-0.3cm]3}](3){};
  \draw[-] (2) edge[red,thick] (3);
  \draw[-] (4) edge[red,thick] (5);
  \draw[-] (6) edge[red,thick] (1);
  \draw[-] (1) edge[red,thick] (3);
  \draw[-] (2) edge[red,thick] (4);
  \draw[-] (3) edge[red,thick] (5);
  \draw[-] (4) edge[red,thick] (6);
  \draw[-] (5) edge[red,thick] (1);
  \draw[-] (6) edge[red,thick] (2);
  \draw[-] (1) edge[red,thick] (4);
  \draw[-] (2) edge[red,thick] (5);
  \draw[-] (3) edge[red,thick] (6);
  \end{scope}
\end{tikzpicture}
\caption{Uniform coloured graphs on $6$ vertices where the linear part of $I(\mathcal L^{-1})$ is induced by symmetries. From left to right: the cycle $C_6$, the complete graph $K_6$, the complete bipartite graph $K_{3,3}$, and the hyperoctahedral graph $H_3$.}\label{families}
\end{figure}
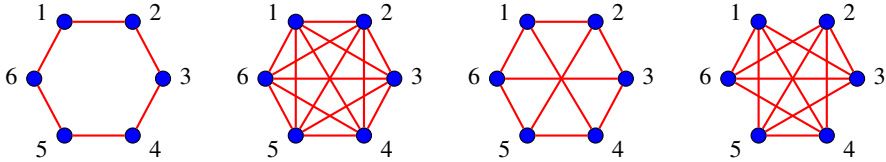

\begin{thm}\label{circulant-families}
Let $G$ be one of: the cycle $C_n$, the complete graph $K_n$, the complete bipartite graph $K_{m,m}$, or the hyperoctahedral graph $H_m$. Then the linear part $L$ of $I(\mathcal L^{-1})$ is induced by symmetries. 
More precisely,

1. If $G = C_n$, then $r=s=\lfloor n/2\rfloor + 1$ and $L$ is generated by
\[
    x_{1,1+d}-x_{i,i+d}, \quad i=2,\dotsc,n,\quad d=0,\dotsc,\lfloor n/2\rfloor,
\]
where all indices are taken modulo $n$ and $x_{ji}$ for $j>i$ is taken to mean $x_{ij}$.

2. If $G = K_n$, then $r = s = 2$ and $L$ is generated by
\begin{align*}
x_{11}&-x_{ii}, \quad i\in V\\
x_{12}&-x_{ij}, \quad i\in E.
\end{align*}

3. If $G = K_{m,m}$, then $r = s = 3$ and $L$ is generated by
\begin{align*}
x_{11}&-x_{ii}, \quad i\in V\\
x_{12}&-x_{ij}, \quad (i,j)\in E\\
x_{13}&-x_{ij}, \quad (i,j)\in E^c.
\end{align*}

4. If $G = H_m$, then $r = s = 3$ and $L$ is generated by
\begin{align*}
x_{11}&-x_{ii},\quad i\in V\\
x_{13}&-x_{ij},\quad (i,j)\in E\\
x_{12}&-x_{ij},\quad (i,j)\in E^c.
\end{align*}

\end{thm}
\begin{proof}
First, all statements about the number of eigenvalues $r$ are proven in~\cite[pp.\ 17, 50]{biggs}.
By Corollary~\ref{action}, giving the linear forms induced by symmetries amounts to giving a full set of representatives for the action of $\Aut(G)$ on
$V*V$.
Since this action decomposes into an action on the three disjoint subsets $V,E$, and $E^c$, it is useful to consider these separately.

1. If $G = C_n$, then $\Aut(G)$ acts transitively on $V$ and $E$. The orbits of the action on $E^c$ are the sets $E^c_d \coloneqq \{(i,i+d) \mid i\in V\}$ for $d = 2,\dotsc, \lfloor n/2\rfloor$.

2. If $G = K_n$, then $\Aut(G)$ acts transitively on $V$ and $E$, and $E^c = \emptyset$.

3. If $G = K_{m,n}$ is any complete bipartite graph, then the complement $G^c$ is the disjoint union of the two complete graphs $K_m = (V_1,E_1^c)$ and $K_n = (V_2, E_2^c)$. The group $\Aut(G)$ acts transitively on $V_1, V_2, E, E^c_1$, and $E^c_2$. In the case $n=m$, there exists an additional symmetry that interchanges $V_1$ and $V_2$. Hence in this case, we see that $\Aut(G)$ acts transitively on $V,E$, and $E^c$.

4. The hyperoctahedral graph $G = H_m$ is the complete multipartite graph on $m$ two-element sets. A similar argument to the one for $K_{m,m}$ shows that $\Aut(G)$ acts transitively on $V,E,$ and $E^c$.

Taking these statements about the action of $\Aut(G)$ on
$V*V$
and applying Corollary~\ref{action} proves all statements in the theorem.
\end{proof}

\begin{exa}
The linear space associated to the uniform coloured $6$-cycle is
\begin{align*}
\mathcal{L}= \left\{\begin{pmatrix}
\lambda_2 & \lambda_1 &0&0&0& \lambda_1\\
\lambda_1 & \lambda_2 & \lambda_1 &0&0&0\\
0&\lambda_1 & \lambda_2 & \lambda_1 &0&0\\
0&0& \lambda_1 & \lambda_2 & \lambda_1 &0\\
0&0&0& \lambda_1 & \lambda_2 & \lambda_1\\
\lambda_1 &0&0&0& \lambda_1 & \lambda_2\\
\end{pmatrix} : \lambda_1,\lambda_2\in \mathbb{C}  \right\}. 
\end{align*}
The spectrum of the (uncoloured) $6$-cycle is $(1_2,-1_2,2,-2)$. There are $4$ orbits of the action of $\Aut(G)$ on
$V*V$,
namely $V,E,E_2^c$ and $E_3^c$.
According to Theorem~\ref{circulant-families}, the following $17$ linear forms generate the linear part of $I(\mathcal L^{-1})$:
\begin{center}
\begin{tabular}{CCCC}
x_{11} - x_{22} & x_{12} - x_{23} & x_{13} - x_{24} & x_{14} - x_{25}\\
x_{11} - x_{33} & x_{12} - x_{34} & x_{13} - x_{35} & x_{14} - x_{36}\\
x_{11} - x_{44} & x_{12} - x_{45} & x_{13} - x_{46} & \\
x_{11} - x_{55} & x_{12} - x_{56} & x_{13} - x_{15} & \\
x_{11} - x_{66} & x_{12} - x_{16} & x_{13} - x_{26} &
\end{tabular}
\end{center}
The graph $G'$ associated to the symmetries of $G$ can be found in Table~\ref{table4}.
\end{exa}

We conclude this section with two families of graphs where the linear part of $I(\mathcal L^{-1})$ is not induced by symmetries, illustrated in Figure~\ref{noncirculantfamilies}. We use the convention that for $m<n$, the vertex set of the complete bipartite graph $K_{m,n}$ is partitioned as $\{1,\dotsc, m\}$ and $\{m+1,\dotsc,m+n\}$.

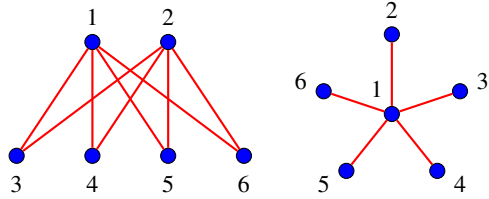
\begin{figure}[H]
\centering
\begin{tikzpicture}
\begin{scope}[scale=1]
\footnotesize
    \draw (1,0) node[draw,circle,inner sep=2pt,fill=blue, label={[xshift=0cm,yshift=-0.7cm]5}](5){};
  \draw (0,1.5) node[draw,circle,inner sep=2pt,fill=blue, label={[xshift=0cm,yshift=0cm]1}](1){};
  \draw (0,0) node[draw,circle,inner sep=2pt,fill=blue, label ={[xshift=0cm, yshift=-0.7cm]4}](4){};
  \draw (1,1.5) node[draw,circle,inner sep=2pt,fill=blue, label ={[xshift=0cm,yshift=0cm]2}](2){};
  \draw (2,0) node[draw,circle,inner sep=2pt,fill=blue, label ={[xshift=0cm, yshift=-0.7cm]6}](6){};
  \draw (-1,0) node[draw,circle,inner sep=2pt,fill=blue, label ={[xshift=0cm,yshift=-0.7cm]3}](3){};
  \draw[-] (1) edge[red,thick] (3);
  \draw[-] (2) edge[red,thick] (3);
  \draw[-] (1) edge[red,thick] (4);
  \draw[-] (2) edge[red,thick] (4);
  \draw[-] (1) edge[red,thick] (5);
  \draw[-] (2) edge[red,thick] (5);
  \draw[-] (1) edge[red,thick] (6);
  \draw[-] (2) edge[red,thick] (6);
  \end{scope}
\end{tikzpicture}\hspace{1em}
\begin{tikzpicture}
\begin{scope}[scale=1.5]
\footnotesize
    \draw (0.2,0) node[draw,circle,inner sep=2pt,fill=blue, label={[xshift=-0.3cm,yshift=-0.5cm]5}](5){};
  \draw (0,0.7) node[draw,circle,inner sep=2pt,fill=blue, label={[xshift=-0.3cm,yshift=-0.2cm]6}](6){};
  \draw (1,0) node[draw,circle,inner sep=2pt,fill=blue, label ={[xshift=0.3cm, yshift=-0.5cm]4}](4){};
  \draw (1.2,0.7) node[draw,circle,inner sep=2pt,fill=blue, label ={[xshift=0.3cm,yshift=-0.2cm]3}](3){};
  \draw (0.6,1.2) node[draw,circle,inner sep=2pt,fill=blue,label={2}](2){};
  \draw (0.6,0.5) node[draw,circle,inner sep=2pt,fill=blue,label={[xshift=-0.2cm, yshift=0cm]1}](1){};
  \draw[-] (1) edge[red,thick] (2);
  \draw[-] (1) edge[red,thick] (3);
  \draw[-] (1) edge[red,thick] (4);
  \draw[-] (1) edge[red,thick] (5);
  \draw[-] (1) edge[red,thick] (6);
  \end{scope}
 \end{tikzpicture}\\
\caption{
The uniform coloured complete bipartite graph $K_{2,4}$
and the star $K_{1,5}$. For these graphs, the linear part of $I(\mathcal L^{-1})$ is not induced by symmetries.
}\label{noncirculantfamilies}
\end{figure}

\begin{prop}\label{complete-bipartite}
Let $G=K_{m,n}$ be the complete bipartite graph with $1<m<n$. Then $r=3$, $s=5$ and the linear part of $I(\mathcal L^{-1})$ is generated by
\begin{align*}
x_{11}-x_{ii},&\quad i\in V_1\\
x_{m+1,m+1}-x_{ii},&\quad i\in V_2\\
x_{12}-x_{ij},&\quad (i,j)\in E^c_1\\
x_{m+1,m+2}-x_{ij},&\quad (i,j)\in E^c_2\\
x_{1,m+1}-x_{ij},&\quad (i,j)\in E\\
m x_{12} - n x_{m+n-1,m+n},& \\
mx_{11}-(n-m) x_{m+n-1,m+n} - mx_{m+n,m+n},&
\end{align*}
where $(V_1,E_1^c)$ and $(V_2,E_2^c)$ are the two 
components of the complement $G^c$.
\end{prop}
\begin{proof}
The statement about $r$ is proven in~\cite[pg.\ 50]{biggs}. The first five sets of linear forms are the forms induced by symmetries, as explained in the proof of Theorem~\ref{circulant-families}. Furthermore, all the above linear forms are linearly independent.

To see that the last two linear forms vanish on $\mathcal L^{-1}$, let $D_{m,n}$ be the determinant of the coloured adjacency matrix $A$ of $K_{m,n}$. By manually calculating the entries of $A^{-1}$ (up to a factor of $\det(A)^{-1}$) we have
\begin{align*}
x_{11} &= D_{m-1,n}, \\
x_{12} &= -(D_{m-1,n} - \lambda_1 D_{m-2,n}), \\
x_{m+n-1,m+n} &= -(D_{m,n-1} - \lambda_1 D_{m,n-2}), \\
x_{n+m,n+m} &= D_{m,n-1}.
\end{align*}
Furthermore, one can verify that $D_{m,n} = \lambda_1^{m+n} - mn\lambda_1^{m+n-2}\lambda_2^2.$ Substituting this formula for $D_{m,n}$ in the above expressions for the entries of $A^{-1}$ shows that the last two linear forms above indeed vanish. The assumption that $m>1$ was used in the middle two equations for the entries of $A^{-1}$.
\end{proof}

\begin{prop}\label{star}
Let $G$ be the \emph{star graph} on $n$ vertices, i.e.\ the complete bipartite graph $K_{1,n-1}$. Then $r = 3$, $s = 4$, and the linear part of $I(\mathcal L^{-1})$ is generated by
\begin{align*}
x_{22} - x_{ii},& \quad i = 3,\dotsc,n \\
x_{12} - x_{ij},& \quad (i,j) \in E\\
x_{23} - x_{ij},& \quad (i,j) \in E^c\\
x_{11} - (n-2)x_{n-1,n} - x_{nn}.&
\end{align*}
\end{prop}
\begin{proof}
The action of $\Aut(G)$ on $V$ has two orbits, one with only one element, so that it does not give any linear form. The action on $E$ and $E^c$ is transitive. This explains the first three sets of linear forms, which are induced by symmetries.

The missing linear form can be derived by computing three entries of the inverse $A^{-1}$. Up to a factor of $\det(A)^{-1}$, we have
\begin{align*}
x_{11} &= \lambda_1^{n-1}, \\
x_{nn} &= \lambda_1^{n-1} - (n-2)\lambda_1^{n-3}\lambda_2^2, \\
x_{n-1,n} &= \lambda_1^{n-3}\lambda_2^2.
\end{align*}
The middle equality holds since the determinant $D_n$ of the coloured adjacency matrix of the star on $n$ vertices satisfies $D_n = \lambda_1^n - (n-1)\lambda_1^{n-2}\lambda_2^2$.
\end{proof}

\section{Open Problems}
By Proposition~\ref{symmetries}, the symmetries of a coloured graph explain some binomial linear forms in the ideal $I(\mathcal{L}^{-1})$.
The following example shows that in general, not all binomial linear forms in $I(\mathcal L^{-1})$ are induced by symmetries.
\begin{exa}\label{example2}
The graph
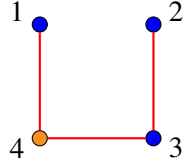
\begin{center}
\begin{tikzpicture}
\begin{scope}[scale=1.5]
\draw (0,0) node[draw,circle,inner sep=2pt,fill=BurntOrange, label={[xshift=-0.3cm,yshift=-0.5cm]4}](4){};
\draw (0,1) node[draw,circle,inner sep=2pt,fill=blue, label={[xshift=-0.3cm,yshift=-0.2cm]1}](1){};
\draw (1,0) node[draw,circle,inner sep=2pt,fill=blue, label ={[xshift=0.3cm, yshift=-0.5cm]3}](3){};
\draw (1,1) node[draw,circle,inner sep=2pt,fill=blue, label ={[xshift=0.3cm,yshift=-0.2cm]2}](2){};
\draw[-] (2) edge[red,thick] (3);
\draw[-] (3) edge[red,thick] (4);
\draw[-] (1) edge[red,thick] (4);
\end{scope}
\end{tikzpicture}
\end{center}
has no nontrivial symmetries, but all elements $X\in \mathcal L^{-1}$ satisfy $x_{13} - x_{24} = 0$.
\end{exa}
However, after computing the linear part of $I(\mathcal L^{-1})$ for all colourings on the 3,4, and 5-cycle where all vertices have the same colour, we arrived at the following conjecture.
\begin{conj}
Let $G$ be a coloured $n$-cycle where all vertices have the same colour. All \emph{binomial} linear forms in $I(\mathcal{L}^{-1})$ are induced by symmetries.
\end{conj}

In Theorem~\ref{circulant-families} we give four families of graphs where the linear part of $I(\mathcal L^{-1})$ is induced by symmetries. In particular, this gives a complete description of the linear part for the uniform coloured $n$-cycle. However, for every $n$ this only represents one possible colouring.

\begin{que}
Can one give a a complete description of the linear part of $I(\mathcal L^{-1})$ for \emph{all} colourings of the $n$-cycle?
\end{que}

Tables~\ref{table2} and \ref{table3} give a complete list of all coloured $5$-cycles with only one vertex colour. Just as in Example~\ref{five-cycle}, some linear forms are not induced by symmetries. We have highlighted these.

All graphs covered by Theorem~\ref{circulant-families} are \emph{circulant graphs}. By definition, these are graphs that have a symmetry which is a cyclic permutation of its vertices. We conjecture that like the families in the theorem, the linear part of $I(\mathcal L^{-1})$ is induced by symmetries for \emph{all} uniform coloured circulant graphs. This is equivalent to $r = s$ in the notation of Section~\ref{uniform-coloured-graphs}:

\begin{conj}
If $G$ is an uncoloured circulant graph, then the number of orbits of the action of $\Aut(G)$ on
the set $V*V$ of unordered pairs of vertices of $G$
coincides with the number of distinct eigenvalues of its adjacency matrix.
\end{conj}

In this article we have focused on linear forms. We do not know whether the quadratic forms of a minimal generating set of $I(\mathcal L^{-1})$ can be deduced graphically from $G$. In Table~\ref{table4} we give these quadratic forms when $G$ is the uniform coloured $n$-cycle for $n=3,\dotsc,8$, along with the graph $G'$ associated to the symmetries of $G$.
 
\paragraph{Acknowledgements.} We thank Shiyue Li, Aida Maraj and Bernd Sturmfels for helpful conversations in the early stages of this project. Thanks to the anonymous referee for encouraging us to broaden the scope of our first version.

\begin{table}[htbp]
\begin{center} 
\begin{tabular}{|>{\centering\arraybackslash}m{3cm}|>{\centering\arraybackslash}m{5cm}|>{\centering\arraybackslash}m{3cm}|} \hline $G$ & Linear Forms & $G'$\\
 \hline \begin{tikzpicture}
\begin{scope}
\footnotesize
  \draw (0.2,0) node[draw,circle,inner sep=2pt,fill=blue, label={[xshift=-0.3cm,yshift=-0.5cm]5}](5){};
  \draw (0,0.7) node[draw,circle,inner sep=2pt,fill=blue, label={[xshift=-0.3cm,yshift=-0.2cm]1}](1){};
  \draw (1,0) node[draw,circle,inner sep=2pt,fill=blue, label ={[xshift=0.3cm, yshift=-0.5cm]4}](4){};
  \draw (1.2,0.7) node[draw,circle,inner sep=2pt,fill=blue, label ={[xshift=0.3cm,yshift=-0.2cm]3}](3){};
  \draw (0.6,1.2) node[draw,circle,inner sep=2pt,fill=blue,label={2}](2){};
  \draw[-] (1) edge[red,thick] (2);
  \draw[-] (2) edge[red,thick] (3);
  \draw[-] (3) edge[red,thick] (4);
  \draw[-] (4) edge[red,thick] (5);
  \draw[-] (5) edge[red,thick] (1);
  \end{scope}
 \end{tikzpicture}
&\footnotesize
\begin{tabular}{CCC}
x_{11}-x_{55}&x_{12}-x_{45}&x_{13}-x_{35}\\
x_{22}-x_{55}&x_{23}-x_{45}&x_{14}-x_{35}\\
x_{33}-x_{55}&x_{34}-x_{45}&x_{24}-x_{35}\\
x_{44}-x_{55}&x_{15}-x_{45}&x_{25}-x_{35}\\
&&\\
\end{tabular}
& \begin{tikzpicture}
\begin{scope}
\footnotesize
    \draw (0.2,0) node[draw,circle,inner sep=2pt,fill=blue, label={[xshift=-0.3cm,yshift=-0.5cm]5}](5){};
  \draw (0,0.7) node[draw,circle,inner sep=2pt,fill=blue, label={[xshift=-0.3cm,yshift=-0.2cm]1}](1){};
  \draw (1,0) node[draw,circle,inner sep=2pt,fill=blue, label ={[xshift=0.3cm, yshift=-0.5cm]4}](4){};
  \draw (1.2,0.7) node[draw,circle,inner sep=2pt,fill=blue, label ={[xshift=0.3cm,yshift=-0.2cm]3}](3){};
  \draw (0.6,1.2) node[draw,circle,inner sep=2pt,fill=blue,label={2}](2){};
  \draw[-] (1) edge[red,thick] (2);
  \draw[-] (2) edge[red,thick] (3);
  \draw[-] (3) edge[red,thick] (4);
  \draw[-] (4) edge[red,thick] (5);
  \draw[-] (5) edge[red,thick] (1);
  \draw[-] (1) edge[ForestGreen,thick] (3);
  \draw[-] (2) edge[ForestGreen,thick] (4);
  \draw[-] (3) edge[ForestGreen,thick] (5);
  \draw[-] (4) edge[ForestGreen,thick] (1);
  \draw[-] (5) edge[ForestGreen,thick] (2);
  \end{scope}
 \end{tikzpicture}\\
\hline
 
\begin{tikzpicture}
\begin{scope}
\footnotesize
    \draw (0.2,0) node[draw,circle,inner sep=2pt,fill=blue, label={[xshift=-0.3cm,yshift=-0.5cm]5}](5){};
  \draw (0,0.7) node[draw,circle,inner sep=2pt,fill=blue, label={[xshift=-0.3cm,yshift=-0.2cm]1}](1){};
  \draw (1,0) node[draw,circle,inner sep=2pt,fill=blue, label ={[xshift=0.3cm, yshift=-0.5cm]4}](4){};
  \draw (1.2,0.7) node[draw,circle,inner sep=2pt,fill=blue, label ={[xshift=0.3cm,yshift=-0.2cm]3}](3){};
  \draw (0.6,1.2) node[draw,circle,inner sep=2pt,fill=blue,label={2}](2){};
  \draw[-] (1) edge[ForestGreen,thick] (2);
  \draw[-] (2) edge[red,thick] (3);
  \draw[-] (3) edge[red,thick] (4);
  \draw[-] (4) edge[red,thick] (5);
  \draw[-] (5) edge[red,thick] (1);
  \end{scope}
 \end{tikzpicture}
&\footnotesize
\begin{tabular}{CC}
 x_{11}-x_{22}&x_{15}-x_{23}\\
x_{33}-x_{55}&x_{34}-x_{45}\\
x_{13}-x_{25}&x_{14}-x_{24}  
\end{tabular}
$\highlight{x_{24}+x_{44}-x_{35}-x_{55}}$
&\begin{tikzpicture}
\begin{scope}
\footnotesize
\draw (0.2,0) node[draw,circle,inner sep=2pt,fill=blue, label={[xshift=-0.3cm,yshift=-0.5cm]5}](5){};
\draw (0,0.7) node[draw,circle,inner sep=2pt,fill=BurntOrange, label={[xshift=-0.3cm,yshift=-0.2cm]1}](1){};
\draw (1,0) node[draw,circle,inner sep=2pt,fill=Purple, label ={[xshift=0.3cm, yshift=-0.5cm]4}](4){};
\draw (1.2,0.7) node[draw,circle,inner sep=2pt,fill=blue, label ={[xshift=0.3cm,yshift=-0.2cm]3}](3){};
\draw (0.6,1.2) node[draw,circle,inner sep=2pt,fill=BurntOrange,label={2}](2){};
\draw[-] (1) edge[OliveGreen,thick] (2);
\draw[-] (2) edge[red,thick] (3);
\draw[-] (3) edge[Goldenrod,thick] (4);
\draw[-] (4) edge[Goldenrod,thick] (5);
\draw[-] (5) edge[red,thick] (1);
\draw[-] (2) edge[YellowGreen,thick] (4);
\draw[-] (3) edge[CornflowerBlue,thick] (5);
\draw[-] (4) edge[YellowGreen,thick] (1);
\draw[-] (3) edge[CarnationPink,thick] (1);
\draw[-] (2) edge[CarnationPink,thick] (5);
\end{scope}
\end{tikzpicture}
\\
 \hline \begin{tikzpicture}
\begin{scope}
\footnotesize
    \draw (0.2,0) node[draw,circle,inner sep=2pt,fill=blue, label={[xshift=-0.3cm,yshift=-0.5cm]5}](5){};
  \draw (0,0.7) node[draw,circle,inner sep=2pt,fill=blue, label={[xshift=-0.3cm,yshift=-0.2cm]1}](1){};
  \draw (1,0) node[draw,circle,inner sep=2pt,fill=blue, label ={[xshift=0.3cm, yshift=-0.5cm]4}](4){};
  \draw (1.2,0.7) node[draw,circle,inner sep=2pt,fill=blue, label ={[xshift=0.3cm,yshift=-0.2cm]3}](3){};
  \draw (0.6,1.2) node[draw,circle,inner sep=2pt,fill=blue,label={2}](2){};
  \draw[-] (1) edge[OliveGreen,thick] (2);
  \draw[-] (2) edge[OliveGreen,thick] (3);
  \draw[-] (3) edge[red,thick] (4);
  \draw[-] (4) edge[red,thick] (5);
  \draw[-] (5) edge[red,thick] (1);
  \end{scope}
 \end{tikzpicture}
 &
 \footnotesize
 \begin{tabular}{CC}
 x_{11}-x_{33}& x_{14}-x_{35}\\
 x_{44}-x_{55}& x_{15}-x_{34}\\
 x_{12}-x_{23}& x_{24}-x_{25}\\
\end{tabular}
{\highlight{x_{13}+x_{34}+x_{55}-x_{33}-x_{35}-x_{45}}} 
 & \begin{tikzpicture}
\begin{scope}
\footnotesize
    \draw (0.2,0) node[draw,circle,inner sep=2pt,fill=BurntOrange, label={[xshift=-0.3cm,yshift=-0.5cm]5}](5){};
  \draw (0,0.7) node[draw,circle,inner sep=2pt,fill=blue, label={[xshift=-0.3cm,yshift=-0.2cm]1}](1){};
  \draw (1,0) node[draw,circle,inner sep=2pt,fill=BurntOrange, label ={[xshift=0.3cm, yshift=-0.5cm]4}](4){};
  \draw (1.2,0.7) node[draw,circle,inner sep=2pt,fill=blue, label ={[xshift=0.3cm,yshift=-0.2cm]3}](3){};
  \draw (0.6,1.2) node[draw,circle,inner sep=2pt,fill=Purple,label={2}](2){};
  \draw[-] (1) edge[OliveGreen,thick] (2);
  \draw[-] (2) edge[OliveGreen,thick] (3);
  \draw[-] (3) edge[red,thick] (4);
  \draw[-] (4) edge[Goldenrod,thick] (5);
  \draw[-] (5) edge[red,thick] (1);
  \draw[-] (1) edge[CornflowerBlue,thick] (3);
  \draw[-] (2) edge[YellowGreen,thick] (4);
  \draw[-] (3) edge[CarnationPink,thick] (5);
  \draw[-] (4) edge[CarnationPink,thick] (1);
  \draw[-] (5) edge[YellowGreen,thick] (2);
  \end{scope}
 \end{tikzpicture}\\
 \hline \begin{tikzpicture}
\begin{scope}
\footnotesize
    \draw (0.2,0) node[draw,circle,inner sep=2pt,fill=blue, label={[xshift=-0.3cm,yshift=-0.5cm]5}](5){};
  \draw (0,0.7) node[draw,circle,inner sep=2pt,fill=blue, label={[xshift=-0.3cm,yshift=-0.2cm]1}](1){};
  \draw (1,0) node[draw,circle,inner sep=2pt,fill=blue, label ={[xshift=0.3cm, yshift=-0.5cm]4}](4){};
  \draw (1.2,0.7) node[draw,circle,inner sep=2pt,fill=blue, label ={[xshift=0.3cm,yshift=-0.2cm]3}](3){};
  \draw (0.6,1.2) node[draw,circle,inner sep=2pt,fill=blue,label={2}](2){};
  \draw[-] (1) edge[OliveGreen,thick] (2);
  \draw[-] (2) edge[red,thick] (3);
  \draw[-] (3) edge[OliveGreen,thick] (4);
  \draw[-] (4) edge[red,thick] (5);
  \draw[-] (5) edge[red,thick] (1);
  \end{scope}
 \end{tikzpicture}
 &\footnotesize\vspace*{-1cm}\begin{tabular}{CC}
x_{11}-x_{44}&x_{13}-x_{24}\\
x_{22}-x_{33}&x_{15}-x_{45}\\
x_{12}-x_{34}&x_{25}-x_{35}   
\end{tabular}
 & \begin{tikzpicture}
\begin{scope}
\footnotesize
    \draw (0.2,0) node[draw,circle,inner sep=2pt,fill=Purple, label={[xshift=-0.3cm,yshift=-0.5cm]5}](5){};
  \draw (0,0.7) node[draw,circle,inner sep=2pt,fill=blue, label={[xshift=-0.3cm,yshift=-0.2cm]1}](1){};
  \draw (1,0) node[draw,circle,inner sep=2pt,fill=blue, label ={[xshift=0.3cm, yshift=-0.5cm]4}](4){};
  \draw (1.2,0.7) node[draw,circle,inner sep=2pt,fill=BurntOrange, label ={[xshift=0.3cm,yshift=-0.2cm]3}](3){};
  \draw (0.6,1.2) node[draw,circle,inner sep=2pt,fill=BurntOrange,label={2}](2){};
  \draw[-] (1) edge[OliveGreen,thick] (2);
  \draw[-] (2) edge[Goldenrod,thick] (3);
  \draw[-] (3) edge[OliveGreen,thick] (4);
  \draw[-] (4) edge[red,thick] (5);
  \draw[-] (5) edge[red,thick] (1);
  \draw[-] (1) edge[CarnationPink,thick] (3);
  \draw[-] (2) edge[CarnationPink,thick] (4);
  \draw[-] (3) edge[YellowGreen,thick] (5);
  \draw[-] (4) edge[CornflowerBlue,thick] (1);
  \draw[-] (5) edge[YellowGreen,thick] (2);
  \end{scope}
 \end{tikzpicture}\\
 \hline \begin{tikzpicture}
\begin{scope}
\footnotesize
    \draw (0.2,0) node[draw,circle,inner sep=2pt,fill=blue, label={[xshift=-0.3cm,yshift=-0.5cm]5}](5){};
  \draw (0,0.7) node[draw,circle,inner sep=2pt,fill=blue, label={[xshift=-0.3cm,yshift=-0.2cm]1}](1){};
  \draw (1,0) node[draw,circle,inner sep=2pt,fill=blue, label ={[xshift=0.3cm, yshift=-0.5cm]4}](4){};
  \draw (1.2,0.7) node[draw,circle,inner sep=2pt,fill=blue, label ={[xshift=0.3cm,yshift=-0.2cm]3}](3){};
  \draw (0.6,1.2) node[draw,circle,inner sep=2pt,fill=blue,label={2}](2){};
  \draw[-] (1) edge[OliveGreen,thick] (2);
  \draw[-] (2) edge[OliveGreen,thick] (3);
  \draw[-] (3) edge[Goldenrod,thick] (4);
  \draw[-] (4) edge[red,thick] (5);
  \draw[-] (5) edge[red,thick] (1);
  \end{scope}
 \end{tikzpicture}
 &\footnotesize None.
 &\begin{tikzpicture}
\begin{scope}
\footnotesize
    \draw (0.2,0) node[draw,circle,inner sep=2pt,fill, label={[xshift=-0.3cm,yshift=-0.5cm]5}](5){};
  \draw (0,0.7) node[draw,circle,inner sep=2pt,fill, label={[xshift=-0.3cm,yshift=-0.2cm]1}](1){};
  \draw (1,0) node[draw,circle,inner sep=2pt,fill, label ={[xshift=0.3cm, yshift=-0.5cm]4}](4){};
  \draw (1.2,0.7) node[draw,circle,inner sep=2pt,fill, label ={[xshift=0.3cm,yshift=-0.2cm]3}](3){};
  \draw (0.6,1.2) node[draw,circle,inner sep=2pt,fill,label={2}](2){};
  \draw[-] (1) edge[thick] (2);
  \draw[-] (2) edge[thick] (3);
  \draw[-] (3) edge[thick] (4);
  \draw[-] (4) edge[thick] (5);
  \draw[-] (5) edge[thick] (1);
  \draw[-] (1) edge[thick] (3);
  \draw[-] (2) edge[thick] (4);
  \draw[-] (3) edge[thick] (5);
  \draw[-] (4) edge[thick] (1);
  \draw[-] (5) edge[thick] (2);
  \end{scope}
 \end{tikzpicture}\\\hline
 \begin{tikzpicture}
\begin{scope}
\footnotesize
    \draw (0.2,0) node[draw,circle,inner sep=2pt,fill=blue, label={[xshift=-0.3cm,yshift=-0.5cm]5}](5){};
  \draw (0,0.7) node[draw,circle,inner sep=2pt,fill=blue, label={[xshift=-0.3cm,yshift=-0.2cm]1}](1){};
  \draw (1,0) node[draw,circle,inner sep=2pt,fill=blue, label ={[xshift=0.3cm, yshift=-0.5cm]4}](4){};
  \draw (1.2,0.7) node[draw,circle,inner sep=2pt,fill=blue, label ={[xshift=0.3cm,yshift=-0.2cm]3}](3){};
  \draw (0.6,1.2) node[draw,circle,inner sep=2pt,fill=blue,label={2}](2){};
  \draw[-] (1) edge[OliveGreen,thick] (2);
  \draw[-] (2) edge[red,thick] (3);
  \draw[-] (3) edge[Goldenrod,thick] (4);
  \draw[-] (4) edge[red,thick] (5);
  \draw[-] (5) edge[red,thick] (1);
  \end{scope}
 \end{tikzpicture}
 &\footnotesize None.
 &\begin{tikzpicture}
\begin{scope}
\footnotesize
    \draw (0.2,0) node[draw,circle,inner sep=2pt,fill, label={[xshift=-0.3cm,yshift=-0.5cm]5}](5){};
  \draw (0,0.7) node[draw,circle,inner sep=2pt,fill, label={[xshift=-0.3cm,yshift=-0.2cm]1}](1){};
  \draw (1,0) node[draw,circle,inner sep=2pt,fill, label ={[xshift=0.3cm, yshift=-0.5cm]4}](4){};
  \draw (1.2,0.7) node[draw,circle,inner sep=2pt,fill, label ={[xshift=0.3cm,yshift=-0.2cm]3}](3){};
  \draw (0.6,1.2) node[draw,circle,inner sep=2pt,fill,label={2}](2){};
  \draw[-] (1) edge[thick] (2);
  \draw[-] (2) edge[thick] (3);
  \draw[-] (3) edge[thick] (4);
  \draw[-] (4) edge[thick] (5);
  \draw[-] (5) edge[thick] (1);
  \draw[-] (1) edge[thick] (3);
  \draw[-] (2) edge[thick] (4);
  \draw[-] (3) edge[thick] (5);
  \draw[-] (4) edge[thick] (1);
  \draw[-] (5) edge[thick] (2);
  \end{scope}
 \end{tikzpicture}\\\hline
  \end{tabular}
  \caption{Linear minimal generators of $I(\mathcal L^{-1})$ for coloured five-cycles. The left column of the table shows the given coloured five-cycle $G$ with all vertices having the same colour. The linear generators of the corresponding reciprocal varieties are shown in the middle column. Highlighted are the generators not induced by symmetries, i.e.\ not found with Proposition~\ref{symmetries}. The right column shows the graph $G'$ associated to the symmetries of $G$.
  If $I'$ is the zero ideal, we give $G'$ as the uncoloured complete graph on $5$ vertices.
  }
  \label{table2}
 \end{center}
\end{table}
 
 \begin{table}[htbp]
\begin{center} 
\begin{tabular}{|>{\centering\arraybackslash}m{3cm}|>{\centering\arraybackslash}m{5cm}|>{\centering\arraybackslash}m{3cm}|} \hline $G$ & Linear Forms & $G'$\\\hline
\begin{tikzpicture}
\begin{scope}
\footnotesize
    \draw (0.2,0) node[draw,circle,inner sep=2pt,fill=blue, label={[xshift=-0.3cm,yshift=-0.5cm]5}](5){};
  \draw (0,0.7) node[draw,circle,inner sep=2pt,fill=blue, label={[xshift=-0.3cm,yshift=-0.2cm]1}](1){};
  \draw (1,0) node[draw,circle,inner sep=2pt,fill=blue, label ={[xshift=0.3cm, yshift=-0.5cm]4}](4){};
  \draw (1.2,0.7) node[draw,circle,inner sep=2pt,fill=blue, label ={[xshift=0.3cm,yshift=-0.2cm]3}](3){};
  \draw (0.6,1.2) node[draw,circle,inner sep=2pt,fill=blue,label={2}](2){};
  \draw[-] (1) edge[red,thick] (2);
  \draw[-] (2) edge[OliveGreen,thick] (3);
  \draw[-] (3) edge[Goldenrod,thick] (4);
  \draw[-] (4) edge[red,thick] (5);
  \draw[-] (5) edge[red,thick] (1);
  \end{scope}
 \end{tikzpicture}
 &\footnotesize None.
 &\begin{tikzpicture}
\begin{scope}
\footnotesize
    \draw (0.2,0) node[draw,circle,inner sep=2pt,fill, label={[xshift=-0.3cm,yshift=-0.5cm]5}](5){};
  \draw (0,0.7) node[draw,circle,inner sep=2pt,fill, label={[xshift=-0.3cm,yshift=-0.2cm]1}](1){};
  \draw (1,0) node[draw,circle,inner sep=2pt,fill, label ={[xshift=0.3cm, yshift=-0.5cm]4}](4){};
  \draw (1.2,0.7) node[draw,circle,inner sep=2pt,fill, label ={[xshift=0.3cm,yshift=-0.2cm]3}](3){};
  \draw (0.6,1.2) node[draw,circle,inner sep=2pt,fill,label={2}](2){};
  \draw[-] (1) edge[thick] (2);
  \draw[-] (2) edge[thick] (3);
  \draw[-] (3) edge[thick] (4);
  \draw[-] (4) edge[thick] (5);
  \draw[-] (5) edge[thick] (1);
  \draw[-] (1) edge[thick] (3);
  \draw[-] (2) edge[thick] (4);
  \draw[-] (3) edge[thick] (5);
  \draw[-] (4) edge[thick] (1);
  \draw[-] (5) edge[thick] (2);
  \end{scope}
 \end{tikzpicture}\\ \hline

 \begin{tikzpicture}
\begin{scope}
\footnotesize
    \draw (0.2,0) node[draw,circle,inner sep=2pt,fill=blue, label={[xshift=-0.3cm,yshift=-0.5cm]5}](5){};
  \draw (0,0.7) node[draw,circle,inner sep=2pt,fill=blue, label={[xshift=-0.3cm,yshift=-0.2cm]1}](1){};
  \draw (1,0) node[draw,circle,inner sep=2pt,fill=blue, label ={[xshift=0.3cm, yshift=-0.5cm]4}](4){};
  \draw (1.2,0.7) node[draw,circle,inner sep=2pt,fill=blue, label ={[xshift=0.3cm,yshift=-0.2cm]3}](3){};
  \draw (0.6,1.2) node[draw,circle,inner sep=2pt,fill=blue,label={2}](2){};
  \draw[-] (1) edge[OliveGreen,thick] (2);
  \draw[-] (2) edge[OliveGreen,thick] (3);
  \draw[-] (3) edge[red,thick] (4);
  \draw[-] (4) edge[Goldenrod,thick] (5);
  \draw[-] (5) edge[red,thick] (1);
  \end{scope}
 \end{tikzpicture}
 &\footnotesize\vspace*{-1cm}\begin{tabular}{CC}
 x_{11}-x_{33}&x_{14}-x_{35}\\
 x_{44}-x_{55}& x_{15}-x_{34}\\
 x_{12}-x_{23}& x_{24}-x_{25}
  \end{tabular}
 & \begin{tikzpicture}
\begin{scope}
\footnotesize
    \draw (0.2,0) node[draw,circle,inner sep=2pt,fill=BurntOrange, label={[xshift=-0.3cm,yshift=-0.5cm]5}](5){};
  \draw (0,0.7) node[draw,circle,inner sep=2pt,fill=blue, label={[xshift=-0.3cm,yshift=-0.2cm]1}](1){};
  \draw (1,0) node[draw,circle,inner sep=2pt,fill=BurntOrange, label ={[xshift=0.3cm, yshift=-0.5cm]4}](4){};
  \draw (1.2,0.7) node[draw,circle,inner sep=2pt,fill=blue, label ={[xshift=0.3cm,yshift=-0.2cm]3}](3){};
  \draw (0.6,1.2) node[draw,circle,inner sep=2pt,fill=Purple,label={2}](2){};
  \draw[-] (1) edge[OliveGreen,thick] (2);
  \draw[-] (2) edge[OliveGreen,thick] (3);
  \draw[-] (3) edge[red,thick] (4);
  \draw[-] (4) edge[Goldenrod,thick] (5);
  \draw[-] (5) edge[red,thick] (1);
  \draw[-] (1) edge[CornflowerBlue,thick] (3);
  \draw[-] (2) edge[YellowGreen,thick] (4);
  \draw[-] (3) edge[CarnationPink,thick] (5);
  \draw[-] (4) edge[CarnationPink,thick] (1);
  \draw[-] (5) edge[YellowGreen,thick] (2);
  \end{scope}
 \end{tikzpicture}\\ \hline
 \begin{tikzpicture}
\begin{scope}
\footnotesize
    \draw (0.2,0) node[draw,circle,inner sep=2pt,fill=blue, label={[xshift=-0.3cm,yshift=-0.5cm]5}](5){};
  \draw (0,0.7) node[draw,circle,inner sep=2pt,fill=blue, label={[xshift=-0.3cm,yshift=-0.2cm]1}](1){};
  \draw (1,0) node[draw,circle,inner sep=2pt,fill=blue, label ={[xshift=0.3cm, yshift=-0.5cm]4}](4){};
  \draw (1.2,0.7) node[draw,circle,inner sep=2pt,fill=blue, label ={[xshift=0.3cm,yshift=-0.2cm]3}](3){};
  \draw (0.6,1.2) node[draw,circle,inner sep=2pt,fill=blue,label={2}](2){};
  \draw[-] (1) edge[OliveGreen,thick] (2);
  \draw[-] (2) edge[Goldenrod,thick] (3);
  \draw[-] (3) edge[red,thick] (4);
  \draw[-] (4) edge[red,thick] (5);
  \draw[-] (5) edge[red,thick] (1);
  \end{scope}
 \end{tikzpicture}
 & \footnotesize $\highlight{x_{14}+x_{44}-x_{35}-x_{55}}$
 & \begin{tikzpicture}
\begin{scope}
\footnotesize
    \draw (0.2,0) node[draw,circle,inner sep=2pt,fill, label={[xshift=-0.3cm,yshift=-0.5cm]5}](5){};
  \draw (0,0.7) node[draw,circle,inner sep=2pt,fill, label={[xshift=-0.3cm,yshift=-0.2cm]1}](1){};
  \draw (1,0) node[draw,circle,inner sep=2pt,fill, label ={[xshift=0.3cm, yshift=-0.5cm]4}](4){};
  \draw (1.2,0.7) node[draw,circle,inner sep=2pt,fill, label ={[xshift=0.3cm,yshift=-0.2cm]3}](3){};
  \draw (0.6,1.2) node[draw,circle,inner sep=2pt,fill,label={2}](2){};
  \draw[-] (1) edge[thick] (2);
  \draw[-] (2) edge[thick] (3);
  \draw[-] (3) edge[thick] (4);
  \draw[-] (4) edge[thick] (5);
  \draw[-] (5) edge[thick] (1);
  \draw[-] (1) edge[thick] (3);
  \draw[-] (2) edge[thick] (4);
  \draw[-] (3) edge[thick] (5);
  \draw[-] (4) edge[thick] (1);
  \draw[-] (5) edge[thick] (2);
  \end{scope}
 \end{tikzpicture}\\\hline
 
 \begin{tikzpicture}
\begin{scope}
\footnotesize
    \draw (0.2,0) node[draw,circle,inner sep=2pt,fill=blue, label={[xshift=-0.3cm,yshift=-0.5cm]5}](5){};
  \draw (0,0.7) node[draw,circle,inner sep=2pt,fill=blue, label={[xshift=-0.3cm,yshift=-0.2cm]1}](1){};
  \draw (1,0) node[draw,circle,inner sep=2pt,fill=blue, label ={[xshift=0.3cm, yshift=-0.5cm]4}](4){};
  \draw (1.2,0.7) node[draw,circle,inner sep=2pt,fill=blue, label ={[xshift=0.3cm,yshift=-0.2cm]3}](3){};
  \draw (0.6,1.2) node[draw,circle,inner sep=2pt,fill=blue,label={2}](2){};
  \draw[-] (1) edge[OliveGreen,thick] (2);
  \draw[-] (2) edge[Goldenrod,thick] (3);
  \draw[-] (3) edge[CarnationPink,thick] (4);
  \draw[-] (4) edge[red,thick] (5);
  \draw[-] (5) edge[red,thick] (1);
  \end{scope}
 \end{tikzpicture}
 &\footnotesize None.
 & \begin{tikzpicture}
\begin{scope}
\footnotesize
    \draw (0.2,0) node[draw,circle,inner sep=2pt,fill, label={[xshift=-0.3cm,yshift=-0.5cm]5}](5){};
  \draw (0,0.7) node[draw,circle,inner sep=2pt,fill, label={[xshift=-0.3cm,yshift=-0.2cm]1}](1){};
  \draw (1,0) node[draw,circle,inner sep=2pt,fill, label ={[xshift=0.3cm, yshift=-0.5cm]4}](4){};
  \draw (1.2,0.7) node[draw,circle,inner sep=2pt,fill, label ={[xshift=0.3cm,yshift=-0.2cm]3}](3){};
  \draw (0.6,1.2) node[draw,circle,inner sep=2pt,fill,label={2}](2){};
  \draw[-] (1) edge[thick] (2);
  \draw[-] (2) edge[thick] (3);
  \draw[-] (3) edge[thick] (4);
  \draw[-] (4) edge[thick] (5);
  \draw[-] (5) edge[thick] (1);
  \draw[-] (1) edge[thick] (3);
  \draw[-] (2) edge[thick] (4);
  \draw[-] (3) edge[thick] (5);
  \draw[-] (4) edge[thick] (1);
  \draw[-] (5) edge[thick] (2);
  \end{scope}
 \end{tikzpicture}\\\hline
  \begin{tikzpicture}
\begin{scope}
\footnotesize
    \draw (0.2,0) node[draw,circle,inner sep=2pt,fill=blue, label={[xshift=-0.3cm,yshift=-0.5cm]5}](5){};
  \draw (0,0.7) node[draw,circle,inner sep=2pt,fill=blue, label={[xshift=-0.3cm,yshift=-0.2cm]1}](1){};
  \draw (1,0) node[draw,circle,inner sep=2pt,fill=blue, label ={[xshift=0.3cm, yshift=-0.5cm]4}](4){};
  \draw (1.2,0.7) node[draw,circle,inner sep=2pt,fill=blue, label ={[xshift=0.3cm,yshift=-0.2cm]3}](3){};
  \draw (0.6,1.2) node[draw,circle,inner sep=2pt,fill=blue,label={2}](2){};
  \draw[-] (1) edge[OliveGreen,thick] (2);
  \draw[-] (2) edge[Goldenrod,thick] (3);
  \draw[-] (3) edge[red,thick] (4);
  \draw[-] (4) edge[CarnationPink,thick] (5);
  \draw[-] (5) edge[red,thick] (1);
  \end{scope}
 \end{tikzpicture}
 &\footnotesize None.
 & \begin{tikzpicture}
\begin{scope}
 \footnotesize    
 \draw (0.2,0) node[draw,circle,inner sep=2pt,fill, label={[xshift=-0.3cm,yshift=-0.5cm]5}](5){};
  \draw (0,0.7) node[draw,circle,inner sep=2pt,fill, label={[xshift=-0.3cm,yshift=-0.2cm]1}](1){};
  \draw (1,0) node[draw,circle,inner sep=2pt,fill, label ={[xshift=0.3cm, yshift=-0.5cm]4}](4){};
  \draw (1.2,0.7) node[draw,circle,inner sep=2pt,fill, label ={[xshift=0.3cm,yshift=-0.2cm]3}](3){};
  \draw (0.6,1.2) node[draw,circle,inner sep=2pt,fill,label={2}](2){};
  \draw[-] (1) edge[thick] (2);
  \draw[-] (2) edge[thick] (3);
  \draw[-] (3) edge[thick] (4);
  \draw[-] (4) edge[thick] (5);
  \draw[-] (5) edge[thick] (1);
  \draw[-] (1) edge[thick] (3);
  \draw[-] (2) edge[thick] (4);
  \draw[-] (3) edge[thick] (5);
  \draw[-] (4) edge[thick] (1);
  \draw[-] (5) edge[thick] (2);
  \end{scope}
 \end{tikzpicture}\\\hline
 \begin{tikzpicture}
\begin{scope}
\footnotesize
    \draw (0.2,0) node[draw,circle,inner sep=2pt,fill=blue, label={[xshift=-0.3cm,yshift=-0.5cm]5}](5){};
  \draw (0,0.7) node[draw,circle,inner sep=2pt,fill=blue, label={[xshift=-0.3cm,yshift=-0.2cm]1}](1){};
  \draw (1,0) node[draw,circle,inner sep=2pt,fill=blue, label ={[xshift=0.3cm, yshift=-0.5cm]4}](4){};
  \draw (1.2,0.7) node[draw,circle,inner sep=2pt,fill=blue, label ={[xshift=0.3cm,yshift=-0.2cm]3}](3){};
  \draw (0.6,1.2) node[draw,circle,inner sep=2pt,fill=blue,label={2}](2){};
  \draw[-] (1) edge[OliveGreen,thick] (2);
  \draw[-] (2) edge[CarnationPink,thick] (3);
  \draw[-] (3) edge[Goldenrod,thick] (4);
  \draw[-] (4) edge[CornflowerBlue,thick] (5);
  \draw[-] (5) edge[red,thick] (1);
  \end{scope}
 \end{tikzpicture}
 &\footnotesize None.
 &\begin{tikzpicture}
\begin{scope}
\footnotesize
    \draw (0.2,0) node[draw,circle,inner sep=2pt,fill, label={[xshift=-0.3cm,yshift=-0.5cm]5}](5){};
  \draw (0,0.7) node[draw,circle,inner sep=2pt,fill, label={[xshift=-0.3cm,yshift=-0.2cm]1}](1){};
  \draw (1,0) node[draw,circle,inner sep=2pt,fill, label ={[xshift=0.3cm, yshift=-0.5cm]4}](4){};
  \draw (1.2,0.7) node[draw,circle,inner sep=2pt,fill, label ={[xshift=0.3cm,yshift=-0.2cm]3}](3){};
  \draw (0.6,1.2) node[draw,circle,inner sep=2pt,fill,label={2}](2){};
  \draw[-] (1) edge[thick] (2);
  \draw[-] (2) edge[thick] (3);
  \draw[-] (3) edge[thick] (4);
  \draw[-] (4) edge[thick] (5);
  \draw[-] (5) edge[thick] (1);
  \draw[-] (1) edge[thick] (3);
  \draw[-] (2) edge[thick] (4);
  \draw[-] (3) edge[thick] (5);
  \draw[-] (4) edge[thick] (1);
  \draw[-] (5) edge[thick] (2);
  \end{scope}
 \end{tikzpicture}\\\hline 
  \end{tabular}
  \caption{Continuation of Table~\ref{table2}.}
  \label{table3}
 \end{center}
\end{table}

\begin{table}[htbp]
\begin{center} 
\begin{tabular}{|>{\centering\arraybackslash}m{0.3cm}|>{\centering\arraybackslash}m{6cm}|>{\centering\arraybackslash}m{3.5cm}|} \hline $n$ & Quadratic Forms & Graph $G'$\\
 \hline
 \footnotesize
 $3$
 &\footnotesize None.
 &\begin{tikzpicture}
    \begin{scope}[scale=1.5]
    \footnotesize
  \draw (0,0) node[draw,circle,inner sep=2pt,fill=blue, label={[xshift=-0.2cm]2}](2){};
  \draw (0,1) node[draw,circle,inner sep=2pt,fill=blue, label={1}](1){};
  \draw (1,0) node[draw,circle,inner sep=2pt,fill=blue, label ={3}](3){};
  \draw[-] (1) edge[red,thick] (2);
  \draw[-] (2) edge[red,thick] (3);
  \draw[-] (1) edge[red,thick] (3);
  \end{scope}
 \end{tikzpicture}\\\hline
\footnotesize 
$4$
 & \footnotesize $x_{13}^{2}-2x_{12}^2+x_{13}x_{11}$
 &
 \begin{tikzpicture}
\begin{scope}[scale=1.5]
\footnotesize
    \draw (0,0) node[draw,circle,inner sep=2pt,fill=blue, label={[xshift=-0.3cm,yshift=-0.5cm]4}](4){};
  \draw (0,1) node[draw,circle,inner sep=2pt,fill=blue, label={[xshift=-0.3cm,yshift=-0.2cm]1}](1){};
  \draw (1,0) node[draw,circle,inner sep=2pt,fill=blue, label ={[xshift=0.3cm, yshift=-0.5cm]3}](3){};
  \draw (1,1) node[draw,circle,inner sep=2pt,fill=blue, label ={[xshift=0.3cm,yshift=-0.2cm]2}](2){};
  \draw[-] (1) edge[red,thick] (2);
  \draw[-] (2) edge[red,thick] (3);
  \draw[-] (3) edge[red,thick] (4);
  \draw[-] (1) edge[red,thick] (4);
  \draw[-] (2) edge[OliveGreen,thick] (4);
  \draw[-] (1) edge[OliveGreen,thick] (3);
  \end{scope}
 \end{tikzpicture}\\\hline
 \footnotesize
 $5$
& \footnotesize $x_{13}^{2}-x_{13}x_{12}-x_{12}^2+x_{13}x_{11}$
& 
\begin{tikzpicture}
\begin{scope}[scale=1.5]
\footnotesize
    \draw (0.2,0) node[draw,circle,inner sep=2pt,fill=blue, label={[xshift=-0.3cm,yshift=-0.5cm]5}](5){};
  \draw (0,0.7) node[draw,circle,inner sep=2pt,fill=blue, label={[xshift=-0.3cm,yshift=-0.2cm]1}](1){};
  \draw (1,0) node[draw,circle,inner sep=2pt,fill=blue, label ={[xshift=0.3cm, yshift=-0.5cm]4}](4){};
  \draw (1.2,0.7) node[draw,circle,inner sep=2pt,fill=blue, label ={[xshift=0.3cm,yshift=-0.2cm]3}](3){};
  \draw (0.6,1.2) node[draw,circle,inner sep=2pt,fill=blue,label={2}](2){};
  \draw[-] (1) edge[red,thick] (2);
  \draw[-] (2) edge[red,thick] (3);
  \draw[-] (3) edge[red,thick] (4);
  \draw[-] (4) edge[red,thick] (5);
  \draw[-] (5) edge[red,thick] (1);
  \draw[-] (1) edge[OliveGreen,thick] (3);
  \draw[-] (2) edge[OliveGreen,thick] (4);
  \draw[-] (3) edge[OliveGreen,thick] (5);
  \draw[-] (4) edge[OliveGreen,thick] (1);
  \draw[-] (5) edge[OliveGreen,thick] (2);
  \end{scope}
 \end{tikzpicture}\\\hline
 \footnotesize
6
&\vspace*{-1cm}\footnotesize\begin{tabular}{C}
2x_{13}^2-x_{12}x_{14}-x_{14}^2\\
2x_{12}x_{13}-x_{11}x_{14}-x_{13}x_{14}\\
2x_{12}^2-2x_{11}x_{13}+x_{12}x_{14}-x_{14}^2
\end{tabular}
&
\begin{tikzpicture}
\begin{scope}[scale=1.5]
\footnotesize
    \draw (0,0) node[draw,circle,inner sep=2pt,fill=blue, label={[xshift=-0.3cm,yshift=-0.5cm]5}](5){};
  \draw (0,1.5) node[draw,circle,inner sep=2pt,fill=blue, label={[xshift=-0.3cm,yshift=-0.2cm]1}](1){};
  \draw (0.9,0) node[draw,circle,inner sep=2pt,fill=blue, label ={[xshift=0.3cm, yshift=-0.5cm]4}](4){};
  \draw (0.9,1.5) node[draw,circle,inner sep=2pt,fill=blue, label ={[xshift=0.3cm,yshift=-0.2cm]2}](2){};
  \draw (-0.4,0.75) node[draw,circle,inner sep=2pt,fill=blue, label ={[xshift=-0.3cm, yshift=-0.3cm]6}](6){};
  \draw (1.3,0.75) node[draw,circle,inner sep=2pt,fill=blue, label ={[xshift=0.3cm,yshift=-0.3cm]3}](3){};
  \draw[-] (1) edge[red,thick] (2);
  \draw[-] (2) edge[red,thick] (3);
  \draw[-] (3) edge[red,thick] (4);
  \draw[-] (4) edge[red,thick] (5);
  \draw[-] (5) edge[red,thick] (6);
  \draw[-] (6) edge[red,thick] (1);
  \draw[-] (1) edge[OliveGreen,thick] (3);
  \draw[-] (2) edge[OliveGreen,thick] (4);
  \draw[-] (3) edge[OliveGreen,thick] (5);
  \draw[-] (4) edge[OliveGreen,thick] (6);
  \draw[-] (5) edge[OliveGreen,thick] (1);
  \draw[-] (6) edge[OliveGreen,thick] (2);
  \draw[-] (1) edge[Goldenrod,thick] (4);
  \draw[-] (2) edge[Goldenrod,thick] (5);
  \draw[-] (3) edge[Goldenrod,thick] (6);
  \end{scope}
\end{tikzpicture}\\\hline
\footnotesize
7
&\vspace*{-1cm}\footnotesize\begin{tabular}{C}
x_{13}^2-x_{12}x_{14}+x_{13}x_{14}-x_{14}^2\\
x_{12}x_{13}-x_{11}x_{14}+x_{12}x_{14}-x_{13}x_{14}\\
x_{12}^2-x_{11}x_{13}+x_{13}x_{14}-x_{14}^2
\end{tabular}
&\begin{tikzpicture}
\begin{scope}[scale=1.5]
\footnotesize
    \draw (0,-0.1) node[draw,circle,inner sep=2pt,fill=blue, label={[xshift=-0.3cm,yshift=-0.5cm]6}](6){};
  \draw (-0.3,1.2) node[draw,circle,inner sep=2pt,fill=blue, label={[xshift=-0.3cm,yshift=-0.2cm]1}](1){};
  \draw (0.8,-0.1) node[draw,circle,inner sep=2pt,fill=blue, label ={[xshift=0.3cm, yshift=-0.5cm]5}](5){};
  \draw (1.1,1.2) node[draw,circle,inner sep=2pt,fill=blue, label ={[xshift=0.3cm,yshift=-0.2cm]3}](3){};
  \draw (-0.4,0.5) node[draw,circle,inner sep=2pt,fill=blue, label ={[xshift=-0.3cm, yshift=-0.3cm]7}](7){};
  \draw (1.2,0.5) node[draw,circle,inner sep=2pt,fill=blue, label ={[xshift=0.3cm,yshift=-0.3cm]4}](4){};
  \draw(0.4,1.6) node[draw,circle,inner sep=2pt,fill=blue, label={[xshift=0cm,yshift=0cm]2}](2){};
  \draw[-] (1) edge[red,thick] (2);
  \draw[-] (2) edge[red,thick] (3);
  \draw[-] (3) edge[red,thick] (4);
  \draw[-] (4) edge[red,thick] (5);
  \draw[-] (5) edge[red,thick] (6);
  \draw[-] (6) edge[red,thick] (7);
  \draw[-] (7) edge[red,thick] (1);
  
  \draw[-] (1) edge[OliveGreen,thick] (3);
  \draw[-] (2) edge[OliveGreen,thick] (4);
  \draw[-] (3) edge[OliveGreen,thick] (5);
  \draw[-] (4) edge[OliveGreen,thick] (6);
  \draw[-] (5) edge[OliveGreen,thick] (7);
  \draw[-] (6) edge[OliveGreen,thick] (1);
  \draw[-] (7) edge[OliveGreen,thick] (2);
  
  \draw[-] (1) edge[Goldenrod,thick] (4);
  \draw[-] (2) edge[Goldenrod,thick] (5);
  \draw[-] (3) edge[Goldenrod,thick] (6);
  \draw[-] (4) edge[Goldenrod,thick] (7);
  \draw[-] (5) edge[Goldenrod,thick] (1);
  \draw[-] (6) edge[Goldenrod,thick] (2);
  \draw[-] (7) edge[Goldenrod,thick] (3);
  \end{scope}
\end{tikzpicture}\\\hline
\footnotesize
8
&\vspace*{-1cm}\footnotesize\begin{tabular}{C}
2x_{14}^2-x_{13}x_{15}-x_{15}^2\\
2x_{13}x_{14}-x_{12}x_{15}-x_{14}x_{15}\\
2x_{12}x_{14}-x_{11}x_{15}-x_{13}x_{15}\\
2x_{13}^2-x_{11}x_{15}-x_{15}^2\\
2x_{12}x_{13}-2x_{11}x_{14}+x_{12}x_{15}-x_{14}x_{15}\\
2x_{12}^2-2x_{11}x_{13}+x_{13}x_{15}-x_{15}^2\\
\\
\end{tabular}
&\begin{tikzpicture}
\begin{scope}[scale=1.5]
\footnotesize
    \draw (0.1,-0.1) node[draw,circle,inner sep=2pt,fill=blue, label={[xshift=-0.3cm,yshift=-0.5cm]6}](6){};
  \draw (0.1,1.6) node[draw,circle,inner sep=2pt,fill=blue, label={[xshift=-0.3cm,yshift=-0.2cm]1}](1){};
  \draw (0.8,-0.1) node[draw,circle,inner sep=2pt,fill=blue, label ={[xshift=0.3cm, yshift=-0.5cm]5}](5){};
  \draw (0.8,1.6) node[draw,circle,inner sep=2pt,fill=blue, label ={[xshift=0.3cm,yshift=-0.2cm]2}](2){};
  \draw (-0.4,1.1) node[draw,circle,inner sep=2pt,fill=blue, label ={[xshift=-0.3cm, yshift=-0.3cm]8}](8){};
  \draw (1.3,1.1) node[draw,circle,inner sep=2pt,fill=blue, label ={[xshift=0.3cm,yshift=-0.3cm]3}](3){};
  \draw (-0.4,0.4) node[draw,circle,inner sep=2pt,fill=blue, label ={[xshift=-0.3cm, yshift=-0.3cm]7}](7){};
  \draw (1.3,0.4) node[draw,circle,inner sep=2pt,fill=blue, label ={[xshift=0.3cm,yshift=-0.3cm]4}](4){};
  \draw[-] (1) edge[red,thick] (2);
  \draw[-] (2) edge[red,thick] (3);
  \draw[-] (3) edge[red,thick] (4);
  \draw[-] (4) edge[red,thick] (5);
  \draw[-] (5) edge[red,thick] (6);
  \draw[-] (6) edge[red,thick] (7);
  \draw[-] (7) edge[red,thick] (8);
  \draw[-] (8) edge[red,thick] (1);
  
  \draw[-] (1) edge[OliveGreen,thick] (3);
  \draw[-] (2) edge[OliveGreen,thick] (4);
  \draw[-] (3) edge[OliveGreen,thick] (5);
  \draw[-] (4) edge[OliveGreen,thick] (6);
  \draw[-] (5) edge[OliveGreen,thick] (7);
  \draw[-] (6) edge[OliveGreen,thick] (8);
  \draw[-] (7) edge[OliveGreen,thick] (1);
  \draw[-] (8) edge[OliveGreen,thick] (2);
  
  \draw[-] (1) edge[Goldenrod,thick] (4);
  \draw[-] (2) edge[Goldenrod,thick] (5);
  \draw[-] (3) edge[Goldenrod,thick] (6);
  \draw[-] (4) edge[Goldenrod,thick] (7);
  \draw[-] (5) edge[Goldenrod,thick] (8);
  \draw[-] (6) edge[Goldenrod,thick] (1);
  \draw[-] (7) edge[Goldenrod,thick] (2);
  \draw[-] (8) edge[Goldenrod,thick] (3);
  
  \draw[-] (1) edge[CarnationPink,thick] (5);
  \draw[-] (2) edge[CarnationPink,thick] (6);
  \draw[-] (3) edge[CarnationPink,thick] (7);
  \draw[-] (4) edge[CarnationPink,thick] (8);
 
  \end{scope}
\end{tikzpicture}\\\hline
  \end{tabular}
  \caption{Quadratic forms in a minimal generating set of $I(\mathcal L^{-1})$ for the first eight uniform coloured $n$-cycles. The right column displays the graph $G'$ representing all linear forms in the minimal generating set (by Theorem~\ref{circulant-families}).}
  \label{table4}
 \end{center}
\end{table}
\newpage
\bibliography{literature}


\end{document}